\documentclass[10pt]{amsart}

\usepackage{amsfonts,amsthm,latexsym,amsmath,amssymb,amscd,amsmath, mathrsfs,color}

\usepackage[cp1251]{inputenc}
\usepackage[russian,english]{babel}

\theoremstyle{plain}
\newtheorem{theorem}{Theorem}[section]

\newtheorem{prepos}[theorem]{Proposition}

\theoremstyle{definition}

\renewcommand{\Im}{\operatorname{Im}}
\renewcommand{\Re}{\operatorname{Re}}

\begin{document}

\author[O.~Katkova]{Olga Katkova}

\address{Department of Science and Mathematics, Wheelock College, USA}
\email{olga.m.katkova@gmail.com }

\author[M.~Tyaglov]{Mikhail Tyaglov}

\address{School of Mathematical Sciences, Shanghai Jiao Tong University\\
and  Faculty of Mathematics, Far East Federal University}
\email{tyaglov@sjtu.edu.cn}

\author[A.~Vishnyakova]{Anna Vishnyakova}

\address{Department of Mechanics \& Mathematics, Kharkov National University}
\email{anna.m.vishnyakova@univer.kharkov.ua}

\title[Linear difference  operators]
{Linear finite difference operators preserving Laguerre-P\'olya class}





\begin{abstract}
We completely describe all finite difference operators of the form
$$
\Delta_{M_1, M_2, h}(f)(z)=M_1(z) f(z+h) + M_2(z) f(z-h)
$$
preserving the Laguerre-P\'olya class of entire functions.
Here $M_1$ and $M_2$ are some complex functions and $h$ is a nonzero
complex number.
\end{abstract}

\maketitle


\setcounter{equation}{0}

\section{Introduction}\label{Section:intro}

In the theory of distribution of zeros of polynomials and transcendental entire functions, one of important problems
is to describe linear transformations preserving the class of real polynomials with real zeros. Hermite
and later Laguerre were, probably, the first who started to study such a problem systematically.
In 1914 P\'olya and Schur~\cite{polsch} completely described the operators acting diagonally on the standard monomial
basis $1$, $x$, $x^2$, \dots of $\mathbb{R}[x]$ and possessing the mentioned preservation property. Later the study of linear
transformations sending real-rooted polynomials to real-rooted polynomials was continued by many authors including
N.\,Obreschkov, S.\,Karlin, B.\,Levin, G.\,Csordas, T.\,Craven, K.\,de Boor, R.\,Varga, A.\,Iserles, S.\,N{\o}rsett, E.\,Saff etc.
Among recent authors it is worth to especially mention P. Br\"{a}nd\'{e}n and J. Borcea~\cite{BrandenBorcea} (see also~\cite{BrandenBorcea2,BrandenBorcea3}) who completely
characterised all linear operators preserving real-rootedness of real polynomials (and some other root location preservers).
Recently P.\,Br\"{a}nd\'{e}n, I.\,Krasikov and B.\,Shapiro~\cite{BKS} made an attempt to transfer the existing
theory of real-rootedness preservers to the basis of Pochhammer symbols and to develop a finite difference
analogue of the P\'olya-Schur theory.

The present work relates to the work~\cite{BKS} and continues it in a very specific way.
Our object of study is the central finite difference operator with
non-constant coefficients
\begin{equation}\label{e1}
\Delta_{M_1, M_2, h}(f)(z) =M_1(z) f(z+h) + M_2(z) f(z-h).
\end{equation}
Here $M_1$ and $M_2$ are certain functions and $h$ is a non-zero \textit{complex} number.
In this work we study operators of the form~\eqref{e1} that preserve real-rootedness
of polynomials. It is clear that the operator $\Delta_{M_1, M_2, h}(p)(z)$ sends an arbitrary polynomial to
a polynomial if, and only if, $M_1$ and $M_2$ are polynomials. The operator~\eqref{e1} and its multiple superpositions were also studied
in~\cite{BKS} in the case when $M_1$ and $M_2$ are polynomials and $h\in\mathbb{R}\setminus\{0\}$. The authors
of~\cite{BKS} established that the operator
$$
T(f)(x) = \sum\limits_{j=0}^k q_j(x) f(x-j)
$$
with $q_j\in\mathbb{C}[x]$, $j=0,1,\ldots,k$, preserves the class of real-rooted polynomials
if and only if  $q_j(x) \not\equiv 0$  for at most one $j$, and $q_j$ has all
real zeros for such an $j$. Consequently, the case of real step $h$ cannot
provide a non-trivial operators of the form~\eqref{e1} preserving real-rootedness of polynomials.
Here we completely describe all such non-trivial operators establishing the following fact.
\begin{prepos}\label{Proposition1}
Let $M_1$ and $M_2$ be polynomials not identically zero. Then for every polynomial $p\in\mathbb{R}[x]$ with only real zeros, the polynomial
$$
M_1(z) p(z+h) + M_2(z) p(z-h)
$$
has only real zeros if, and only if, $\Re h=0$, and either $\left|\dfrac{M_1(z)}{M_2(z)}\right|\equiv1$, or $M_1(z)=e^{i\theta}\cdot\overline{M_2(\bar{z})}$, $\theta\in[0,2\pi)$,
and all the zeros of the polynomial $M_2$ lie in the half-plane $\Im h\cdot\Im z\geqslant0$.
\end{prepos}
Recall that polynomials with all zeros in the closed upper (or lower) half-plane are called \textit{quasi-Hermite-Biehler} polynomials
in analogue with quasi-stable polynomials whose zeros lie on the closed left-half plane. One direction of Proposition~\ref{Proposition1}
is easy. Indeed, the case $\Re h=0$, $|M_1(z)/M_2(z)|\equiv1$ was established in~\cite{pw}, but if
$\Re h=0$, $M_1(z)=e^{i\theta}\overline{M_2(\bar{z})}$, and all zeros of $M_2$ lie in the half-plane $\Im h\cdot\Im z\geqslant0$, then
$$
e^{-i \frac{\theta } {2}} (M_1(z) p(z+h) + M_2(z) p(z-h) =
2 \sum_{k=0}^\infty c_k z^k
$$
whenever
$$
 e^{-i \frac{\theta } {2}} M_2(z) p(z-h) = \sum_{k=0}^\infty
(c_k + i d_k) z^k.
$$
Since $e^{-i \frac{\theta } {2}}M_2(z) p(z-h)$ has all zeros in the upper half-plane ($\Im h>0$) or in the lower half-plane ($\Im h<0$), the polynomial $\sum\limits_{k=0}^\infty c_k z^k$ has only real zeros by Hermite-Biehler
theorem~\cite{her,bieh} (see also~\cite{KreinNaim} or~\cite[Chapter VII]{lev}).

The proof of the converse statement is more difficult, but we prove it here as a particular case of a more general
statement. In the present work, we extend Proposition~\ref{Proposition1} to  entire functions.
Note that some special transcendental functions appear to be solutions of certain finite difference operators of a form
similar to~\eqref{e1}. Riemann's functional equation for the Riemann $\zeta$-function~\cite{Riemann} is such an example.
Some special functions can be results of the action of operators of the form~\eqref{e1}. Thus, it is worth to
extend the domain of the operator~\eqref{e1} from polynomials to entire functions. A natural extension of
polynomials with real roots is the so-called Laguerre-P\'olya class.

A real entire function $f$ is said to be in the {\it
Laguerre-P\'olya class}, $f \in \mathcal{L-P}$, if
\begin{equation}\label{e2}
 f(z) = c z^n e^{- a z^2+b z}\prod_{k=1}^\infty
\left(1-\frac {z}{x_k} \right)e^{\tfrac {z}{x_k}},
\end{equation}
where $c, b, x_k \in  \mathbb{R}$, $x_k\ne 0$,  $a \geqslant 0$,
$n$ is a nonnegative integer and $\sum\limits_{k=1}^\infty x_k^{-2} <
\infty$. The product in the right-hand side of~\eqref{e2} can be
finite or empty (in the latter case the product equals 1).

This class is essential in the theory of entire functions due to the fact that these
and only these functions are the uniform limits, on compact subsets of
$ \mathbb{C}$, of polynomials with only real zeros. For various
characterizations of the
Laguerre-P\'olya class see e.g. \cite[p. 100]{pol}, \cite{polsch}  or  \cite[Kapitel II]{O}.
Thus, in this work we study operators of form~\eqref{e1} that preserve the Laguerre-P\'olya
class, that is, $\Delta_{M_1, M_2, h}(\mathcal{L-P})\subset\mathcal{L-P}$.

G.\,P\'olya seems to be the first who obtained some results
in $\mathcal{L-P}$-preserving properties of an operator of the form~\eqref{e1}.
In~\cite{pol1} he established that if
$f \in \mathcal{L-P}$ then
$$
f(x+ic) + f(x-ic)\in \mathcal{L-P}
$$
for every
$c \in \mathbb{R}$. In \cite{pw} it was observed that the same fact is valid
for the transformation
$$
f(x+ic) + \theta f(x-ic), \quad |\theta| = 1.
$$

In this work, we completely describe all the operators of the form~\eqref{e1} that preserve the Laguerre-P\'olya class.
The main result of the work is the following theorem.
\begin{theorem}\label{th:mth2}
Let $M_1$ and $M_2$ be non-zero entire functions. Then
operator~\eqref{e1} preserves the class~$\mathcal{L-P}$ if, and only if, $\Re h=0$ and the functions $M_1$ and $M_2$ satisfy the conditions:
\begin{itemize}
\item[1.] $ M_1(z)=\overline {M_2(\bar{z})} , \   z \in \mathbb{C}$.
\item[2.] The function $M_2$  is of the form
\begin{equation}\label{t4}
M_2(z)=C z^{n} e^{-az^2+bz}\prod_{k=1}^{\infty}
\left(1-\dfrac{z}{\alpha_k}\right)e^{\tfrac{z}{\alpha_k}}\prod_{k=1}^{\infty} \left(1-\dfrac{z}{x_k}\right)e^{\tfrac{z}{x_k}},
\end{equation}
where $\Im\alpha_k \ >0$, $x_k \in \mathbb{R}\setminus \{0\}$, $\sum\limits_{k=1}^{\infty}\left(|\alpha_k|^{-2} + x_k^{-2}\right)<\infty$, $a \geqslant 0$, $C, \ b \in \mathbb{C}$,
$n\in\mathbb{N}\cup\{0\}$,
and
\begin{equation}\label{q1}
\Im b \geqslant 0\quad\text{whenever} \quad  \sum\limits_{k=1}^{\infty} \frac{1}{|\alpha_k|}<\infty,
\end{equation}
or
\begin{equation} \label{q2}
\left(\Im
b -\sum_{k=1}^{\infty} \frac{\Im\alpha_k}{|\alpha_k|^2}\right) \geqslant 0
\quad\text{whenever}\quad \sum\limits_{k=1}^{\infty}\frac{1}{|\alpha_k|}=\infty.
\end{equation}
\end{itemize}
\end{theorem}

\vspace{2mm}

The paper is organized as follows. In Section~\ref{section:premil.res}, we present some preliminary
results such as necessary conditions for the operator~\eqref{e1} to preserve the class~$\mathcal{L-P}$
and simplified version of Theorem~\ref{th:mth2}. Section~\ref{section:proof.thm1.2} is devoted
to the proof of Theorem~\ref{th:mth2}. In Section~\ref{Section:conclusion}, we give some remarks an conclusions.

\setcounter{equation}{0}

\section{Preliminary results}\label{section:premil.res}

In this section, we prove some auxiliary facts that we use for the proof of Theorem~\ref{th:mth2}.
At first, we establish some necessary conditions for the operator~\eqref{e1}
to preserve the Laguerre-P\'olya class.

\begin{theorem}\label{th:mst1}
Let $M_1$ and $M_2$ be two given functions, $M_1, M_2 : \mathbb{C} \to \mathbb{C},$
$M_1 \not\equiv 0, M_2 \not\equiv 0,$  and $h$ be a complex number, $h\ne 0$.  Suppose
that the linear operator $\Delta_{M_1, M_2,   h}$ of the
form~\eqref{e1} has the property  $\Delta_{M_1, M_2,   h}(\mathcal{L-P})\subset \mathcal{L-P}$.  Then $M_1$
and $M_2$ are entire  functions and either $h\in\mathbb{R}$ or $i h\in\mathbb{R}$.
\end{theorem}
\begin{proof}
Since for every $b\in\mathbb{R}$ the function $e^{b z}$
belongs to the Laguerre-P\'olya class, the function
$$
T_{M_1,M_2,h}(e^{b z})=M_1(z)e^{z+h}+M_2(z)e^{z-h}
$$
belongs to the same class by assumption. Hence $M_1(z) + e^{-2 bh} M_2(z)$ is an entire function in the
class~$\mathcal{L-P}$ for all real~$b$. By assumption $h\neq0$, so the functions $M_1$ and $M_2$ are
entire.

Furthermore, the functions $f_1(x) =1$ and $f_2(x) =x$ belong to the class $\mathcal{L-P}$, so
$$
M_1(x)+M_2(x) \in \mathbb{R}\quad\text{and}\quad
%
%
x(M_1(x)+M_2(x)) + h(M_1(x) - M_2(x)) \in \mathbb{R}
$$
for all $x\in \mathbb{R}$. Consequently, $ h(M_1(x) - M_2(x)) \in \mathbb{R}$  for all $ x\in \mathbb{R}.$

Additionally, note that $f_3(x) =x^2 \in \mathcal{L-P}$ so
$$
x^2(M_1(x)+M_2(x)) + 2xh(M_1(x) - M_2(x)) +
h^2(M_1(x)+M_2(x)) \in \mathbb{R}
$$
for all $ x\in \mathbb{R}$, whence $h^2(M_1(x) + M_2(x)) \in\mathbb{R}$  for all $ x\in \mathbb{R}$.
Thus, $h^2 \in \mathbb{R} $ whenever $M_1(x) + M_2(x) \not\equiv 0$.

However, if $M_1(x) + M_2(x) \equiv 0$, then applying the operator $T_{M_1,M_2,h}$ to the function
$f_4(x) =x^3 \in \mathcal{L-P}$, we obtain
$$
M_1(x)\left((x+h)^3 - (x-h)^3\right) = h M_1(x) (6x^2  + 2h^2) \in \mathbb{R},\qquad\forall x\in\mathbb{R}.
$$
Since $ h(M_1(x) - M_2(x)) =2hM_1(x) \in \mathbb{R}$,  we get $h^2 \in \mathbb{R}$, as required.
\end{proof}

Thus, to study operators of the form~\eqref{e1} preserving the class $\mathcal{L-P}$, we must
take $M_1$ and $M_2$ to be entire functions and $h\in i\mathbb{R}$ or $h\in\mathbb{R}$. As we mentioned in Section~\ref{Section:intro}, the case
of real $h$ was proved to be trivial in~\cite{BKS}, so in what follows, we consider
the operator:
\begin{equation}\label{e3}
T_{M_1, M_2} (f) (z) =M_1(z) f(z+i) + M_2(z) f(z-i)
\end{equation}
where the functions $M_1$ and $M_2$ are entire.

Every entire function in the Laguerre-P\'olya class has only real zeros or is identically zero. Thus, our first step is to
find conditions when the image $T_{M_1, M_2}(\mathcal{L-P})$ consists of entire functions with such a property.
\begin{theorem}\label{th:mth1}
Let $M$ be a meromorphic function. Then for every function $f\in \mathcal{L-P}$
the function $f(z+i)+M(z)f(z-i)$ has only real roots  or is identically zero if, and only if,  one of the
following two conditions holds:
\begin{itemize}
\item[(i)]  $|M(z)| < 1$ whenever $\Im z>0$, and $|M(z)| > 1$ whenever $\Im z <0$;\\
\item[(ii)]  $M$ is a constant function with $|M(z)| \equiv 1.$
\end{itemize}
\end{theorem}
\begin{proof} Suppose that for every $f \in \mathcal{L-P}$ the function
 $f(z+i)+M(z)f(z-i)$ has only real roots or is identically zero, and let $M \not\equiv const$.
In contradiction to (i) assume that   there exists a point~$z_0$ with $\Im z_0 >0$ such that\footnote{The case   $\Im z_0 < 0$ and $|M(z_0)| \leqslant 1$ can be considered analogously.} $|M(z_0)| \geqslant 1$, and
denote $w_0 := M(z_0).$  Let  $z_0 = \alpha + i \beta ,$ $\alpha \in \mathbb{R},$ $\beta >0,$
and $ w_0 = R e^{i\theta}, $ $R\geqslant 1,  \theta \in \mathbb{R}.$

Consider the function $f(z) =e^{-a z^2 +bz}\in\mathcal{L-P}$ where $a=\frac{\log R}{4 \beta} \geqslant 0$
and $b= 2a \alpha +\theta/2\in \mathbb{R}$. It is easy to check that
 $\frac{f(z_0+i)}{f(z_0-i)} = w_0$, so the equation $f(z+i)+M(z)f(z-i)=0$
has a non-real root~$z_0$. Since the function $f(z+i)+M(z)f(z-i)$
cannot have nonreal zeros by assumption unless it is identically zero, we conclude
that
$$
e^{-a (z+i)^2 +b(z+i)} + M(z) e^{-a (z-i)^2 +b(z-i)}\equiv 0,
$$
so $M(z)=- e^{-4iaz +2bi}$.

Now to come to contradiction, it suffices to show that $a=0$. Suppose this is not true, so $a\neq 0$. Then for the function $f(z)= \left(z^2 - \frac{4}{e^{4a} -1}\right)e^{b z} \in \mathcal{L-P}$ we have that the function
$$
g(z): = f(z+i)+M(z)f(z-i)= \left((z+i)^2 - \frac{4}{e^{4a} -1}\right)e^{b (z+i)}  - e^{-4iaz +2bi}
\left((z-i)^2 - \frac{4}{e^{4a} -1}\right)e^{b (z-i)}
$$
is not identically zero, so it must have only real zeros by assumption. At the same time it is easy to check that $g(i)=0.$
The contradiction implies $a=0$, so $M$ is a constant function with $|M(z)| \equiv 1$. Therefore, if the function
$f(z+i)+M(z)f(z-i)$ belongs to the class $\mathcal{L-P}$ whenever $f\in\mathcal{L-P}$, then $(i)$ or $(ii)$ holds.

Assume now that the function $M(z)$ satisfies the conditions (i) or (ii). Fix a function
$f\in\mathcal{L-P}$, and suppose, on the contrary, that for some $z_0 \in \mathbb{C}$ the following
holds
\begin{equation}\label{th:mth1.proof.0}
f(z_0+i)+M(z_0)f(z_0-i)=0,
\end{equation}
or, equivalently,
$$
\dfrac{f(z_0+i)}{f(z_0-i)} = -M(z_0).
$$

Since the function $f$ can be represented in the form~\eqref{e2}, we have
$$  \left| \frac{f(z_0+i)}{f(z_0-i)}\right| = e^{4a \mbox{Im}\ z_0 } \left|  \frac{z_0 +i}{z_0 -i}\right|^n \prod_{k=1}^\infty
\left| \frac{x_k - z_0 -i}{x_k - z_0 +i} \right|.
$$
It is clear that for any $\lambda\in\mathbb{C}$ with $\Im\lambda>0$
\begin{equation}\label{th:mth1.proof.1}
 \left| \frac{f(\lambda+i)}{f(\lambda-i)}\right| \geqslant 1,
\end{equation}
where the inequality is strict if $f$ has at least one zero.

Analogously, for $\Im\lambda<0$
\begin{equation}\label{th:mth1.proof.2}
 \left| \frac{f(\lambda+i)}{f(\lambda-i)}\right| \leqslant 1,
\end{equation}
where the inequality is strict if $f$ has at least one zero.

If (i) holds, then inequalities~\eqref{th:mth1.proof.1}--\eqref{th:mth1.proof.2} show that
equality~\eqref{th:mth1.proof.0} is impossible unless $z_0 \in \mathbb{R}$, a contradiction.

Suppose now that (ii) holds, and $M(z) \equiv e^{ic}$, $c\in \mathbb{R}$.
If $f(z) = e^{- a z^2 +b z},$ $a\geqslant0$, $b\in \mathbb{R}$, then the function
$$
f(z+i)+M(z)f(z-i)=
e^{-a(z+i)^2 +b(z+i)} \left(1 + e^{i(c+ 4az -2b)}  \right)
$$
has only real zeros or is identically zero.

If $f$ has at least one zero, then by~\eqref{th:mth1.proof.1}--\eqref{th:mth1.proof.2}, the
equality~\eqref{th:mth1.proof.0} implies $z_0 \in \mathbb{R}$. Thus, any zero of
$f(z+i)+M(z)f(z-i)$ is real unless this function is identically zero.
\end{proof}

The following theorem provides necessary and sufficient conditions on the functions
$M_1$ and $M_2$ for the operator $T_{M_1,M_2}$ to preserve the class $\mathcal{L-P}$.

\begin{theorem}\label{th:mth3}
Let $M_1$ and $M_2$ be entire functions not identically zero. Then for every function $f\in \mathcal{L-P}$,
$$
M_1(z)f(z+i)+M_2(z)f(z-i) \in \mathcal{L-P}
$$
if, and only if, $M_1$ and $M_2$ satisfy the conditions:
\begin{itemize}
\item[1)] $M_1(z)=\overline{M_2(\bar{z})}$;
\item[2)] $ \left|\dfrac{M_2(z)}{M_1(z)}\right| < 1$ for every $z$ with $ \mbox{Im}\ z > 0 , $ or
 $\dfrac{M_2(z)}{M_1(z)}$ is a constant function with $ \left|\dfrac{M_2(z)}{M_1(z)}\right| \equiv 1$;
\item[3)] The function $M_2$  is of the form
\end{itemize}
\begin{equation}\label{e4}
M_2(z)=
C z^{n} e^{- a z^2 +b z}
\prod_{k=1}^\infty\left(1-\dfrac{z}{\alpha_k}\right)e^{\tfrac{z}{\alpha_k}},
\end{equation}
where $C, b \in \mathbb{C}, $ $\Im b \geqslant 0$, $n \in \mathbb{N} \cup \{0\}$,
$a \geqslant 0$, $\alpha_k \neq 0$, $\Im\alpha_k \geqslant 0$,
 and $\sum\limits_{k=1}^\infty \dfrac{1}{|\alpha_k|^2} < \infty$.
\end{theorem}
\begin{proof}
First, we suppose that for any $f\in \mathcal{L-P}$, the function $M_1(z)f(z+i)+M_2(z)f(z-i)$ belongs to
$ \mathcal{L-P}$. Then the function $f(z+i)+\dfrac{M_2(z)}{M_1(z)}f(z-i)$
has only real roots or is identically zero whenever $f\in \mathcal{L-P}$.
By Theorem~\ref{th:mth1}, either the function $M(z):=\dfrac{M_2(z)}{M_1(z)}$
satisfies $|M(z)|<1$ for $\Im z>0$ and $|M(z)|>1$ for $\Im z<0$, or $|M(z)|\equiv1$,
so the condition $2)$ holds. This implies
\begin{equation}\label{a1}
\left|\frac{M_2(x)}{M_1(x)}\right|=1 \quad\text{for}\quad x \in\mathbb{R},
\end{equation}
and by reflection principle, we have
\begin{equation}\label{a2}
M_2(\alpha)=0 \Longrightarrow M_1(\bar{\alpha})=0.
\end{equation}
Thus, all the common zeros of $M_1(z)$ and $M_2(z)$ are real, and
\begin{equation}\label{d4}
M_2(\alpha)=0 \Longrightarrow \Im\alpha\geqslant 0.
\end{equation}

Furthermore, since the function $f(z)\equiv1$ is in the class $\mathcal{L-P}$, we get $M_1(z)+M_2(z) \in \mathcal{L-P} $ by assumption. Consequently,
$M_1(x)+M_2(x) \in \mathbb{R}$ for $ x\in\mathbb{R}$, so $\Im M_1(x)=-\Im M_2(x)$ whenever $x\in \mathbb{R}$. Moreover,  the function $f(z)=z$ is in the class  $\mathcal{L-P}$ too, so $M_1(z)(z+i)+M_2(z)(z-i) \in \mathcal{L-P}$ by assumption. This implies that $i(M_1(x)-M_2(x)) \in \mathbb{R}$ for $x\in\mathbb{R}$,  or, equivalently,
$\Re M_1(x)=\Re M_2(x)$ for $x\in \mathbb{R}.$  Thus, $ M_1(x)=\overline {M_2(x)}$
for all $x  \in \mathbb{R}$ that implies $ M_1(z)=\overline{M_2(\bar{z})}$ whenever $z \in \mathbb{C}$. Thus, the condition $1)$ is true.

To prove the condition $3)$ we note first that by~\eqref{a2} all the common zeros of the functions $M_1(z)$ and $M_2(z)$ are real. Moreover, these zeros belong to the zero set of the function $M_1(z)+M_2(z)$ which is in the class $\mathcal{L-P}$ as we mentioned above. So we have
\begin{equation}\label{f20}
\sum_{\alpha \in \mathcal{Z}_M} \frac{1}{\alpha^2}<\infty,
\end{equation}
where $ \mathcal{Z}_M := \{\alpha\in\mathbb{R}\setminus\{0\} \  :\   M_1(\alpha)=0 \wedge M_2(\alpha) =0\}$.

Furthermore, since $ M_1(z)+M_2(z)\in \mathcal{L-P}$, and  $M_1(z)(z+i)+M_2(z)(z-i) \in \mathcal{L-P}$,
the functions~$M_1(z)$ and $M_2(z)$ are of growth order at most $2$ and of normal type. The function $\dfrac{M_2(z)}{M_1(z)}$ is a meromorphic function bounded in the upper half-plane, therefore the non-real zeros of $M_2(z)$ must satisfy the Blaschke condition
for the upper half-plane:
\begin{equation}
\label{f18}\sum_k \frac{\mbox{Im}\  \alpha_k}{|\alpha_k|^2+1}<\infty
\end{equation}
(see, for example, the end of (\cite[Chapter VI]{K1}) or (\cite[Chapter VII]{lev})).

Obviously, there exists $\delta >0$ such that
\begin{equation}\label{f19}
M_2(z) =0\ \wedge \ z \neq 0 \Longrightarrow z \notin \{ z\ : \  |\Im z| \leqslant \delta \wedge
 |\Re z| \leqslant \delta \}.
\end{equation}
Let us split all the non-real zeros of $M_2$ into two groups:
$$
\mathcal{Z}_1=\{\alpha_k:\ M_2(\alpha_k)=0\wedge\Im\alpha_k > \delta \} \quad
\text{and}\quad\mathcal{Z}_2=\{\alpha_k:\ M_2(\alpha_k)=0\wedge0 <\Im \alpha_k\leqslant \delta \},
$$
so we have
$$
\mathcal{Z} := \mathcal{Z}_M \cup \mathcal{Z}_1 \cup \mathcal{Z}_2,
$$
where $\mathcal{Z}$ is the set of all non-zero roots of $M_2$.
From the Blaschke condition~\eqref{f18} it follows that
\begin{equation}\label{f9}
\sum_{\alpha \in \mathcal{Z}_1}\frac{1}{|\alpha|^2}<\infty .
\end{equation}

Now let us note that according to Lindel\"of's Theorem (see e.g. \cite[Chapter
2]{GO}, \cite[ p. 22, Problem 3]{K}), for any entire function $F$ of integral
order $\rho \in \mathbb{N}$ and of normal
(finite) type, the sums
$$
|S(r)|:=\left |\sum_{\{z\,:\, f(z)=0,\, |z|\le r,\,
z\not=0\}}\frac{1}{z^\rho}\right|
$$
are bounded as $r\to \infty$. So for the function $M_2(z)$ whose order of growth
does not exceed $2$, normal type, there exists a constant $C > 0$ such that for
every $R>0$
$$
C \geqslant \left|\sum_ {\alpha \in \mathcal{Z},\ |\alpha|\leq R  } \frac{1}{\alpha^2}\right|\geqslant
\left|\sum_ {\alpha \in \mathcal{Z}_2,\ |\alpha|\leq R  } \frac{1}{\alpha^2}\right| -
\sum_ {\alpha \in \mathcal{Z}_M,\ |\alpha|\leq R  } \left|\frac{1}{\alpha^2}\right| -
\sum_ {\alpha \in \mathcal{Z}_1,\ |\alpha|\leq R  }\left| \frac{1}{\alpha^2}\right|.
$$
This inequality together with \eqref{f20} and \eqref{f9} provide the existence of a constant $K > 0$ such that
$$
K \geqslant \left|\sum_ {\alpha \in \mathcal{Z}_2,\ |\alpha|\leqslant R  } \frac{1}{\alpha^2}\right| \geqslant  \mbox{Re} \left(
\sum_ { \alpha \in \mathcal{Z}_2,\ |\alpha |\leqslant R } \frac{1}{\alpha ^2}\right)  =  \sum_ {\alpha \in \mathcal{Z}_2,\ |\alpha |\leqslant R }
\frac{(\mbox{Re}\  \alpha)^2-(\mbox{Im}\  \alpha)^2}{|\alpha|^4} \geqslant 0
$$
for every $R>0$ (we take into account \eqref{f19}). Thus, we have
$$
K \geq  \sum_ {\alpha \in \mathcal{Z}_2,\ |\alpha |\leq R }
\frac{(\mbox{Re}\  \alpha)^2-(\mbox{Im}\  \alpha)^2}{|\alpha|^4} \geq  \sum_ {\alpha \in \mathcal{Z}_2,\ |\alpha |\leq R }
\frac{|\alpha|^2-2\delta^2}{|\alpha|^4}
$$
Since the order of $M_2$ does not exceed $2$ we get $\sum\limits_{\alpha \in \mathcal{Z}_2}
\frac{1}{|\alpha|^4} < \infty$ that implies
\begin{equation}\label{f10}
\sum_{\alpha \in \mathcal{Z}_2}\frac{1}{|\alpha|^2}<\infty.
\end{equation}
Using (\ref{f20}),  (\ref{f9}) and (\ref{f10}) we have
$$
\sum_{\alpha \in \mathcal{Z}}\frac{1}{|\alpha|^2}<\infty.
$$
Thus, by the Hadamard Factorization Theorem (see \cite[Chapter II]{GO} or\cite[ p.22]{B})  the functions
$M_2$ can be represented in the form
\begin{equation}\label{f6}
M_2(z)= C z^{n} e^{a z^2+bz}\prod_{k=1}^\infty \left(1-\frac{z}{\alpha_k}\right) e^{\tfrac{z}{\alpha_k}},
\end{equation}
where $n\in \mathbb{N} \cup \{0\}$, $C, a, b\in \mathbb{C}$,  $ \alpha_k  \in \mathbb{C}\setminus \{ 0\}$,
$\Im\alpha_k \geq 0$,  and $\sum\limits_{k=1}^\infty \frac{1}{|\alpha_k|^2}<\infty$. Since
$ M_1(z)=\overline {M_2(\bar{z})} , \   z \in \mathbb{C}$ we have
\begin{equation}\label{f6_1}
M_1(z)= \overline{C} z^{n} e^{\bar{a} z^2+\bar{b}z}\prod_{k=1}^\infty \left(1-\frac{z}{\overline{\alpha_k}}\right)
e^{\tfrac{z}{\overline{\alpha_k}}}.
\end{equation}

Let $M(z)$ be an entire function such that
$$
M_2(z) =  e^{a z^2} M(z),\qquad M_1(z) =  e^{\bar{a} z^2}\ \overline{M(\bar{z})}.
$$
Then from the condition $2)$ established above, we have
\begin{equation}\label{f6_2}
\left|\frac{M_2(z)}{M_1(z)} \right|= e^{-4(\Im a)(\Re z)(\Im z)}
\left|\frac{M(z)}{\overline{M(\bar{z})}} \right| \leqslant 1,  \ \ \mbox{ when} \ \ \Im z>0.
\end{equation}
The formul\ae~\eqref{f6}--\eqref{f6_1} imply that the genus of the meromorphic function
$\dfrac{M(z)}{\overline{M(\bar{z})}}$ is at most $1$, so the the inequality~\eqref{f6_2} can be valid only if
\begin{equation}
\label{d1}
\mbox{Im} \ a=0.
\end{equation}
Indeed, if $\Im a>0$ ($\Im a<0$), then for $\Re z = \Im z$ (resp.  $\Re z = - \Im  z$),
$\left|\dfrac{M_2(z)}{M_1(z)} \right|>1$ whenever $|z|$ is sufficiently large. Hence,
$$
M_2(z)=e^{a z^2}M(z), \ M_1(z) = e^{a z^2}\ \overline{M(\bar{z})}, \  a\in\mathbb{R},
$$
where $M$ is an entire functions whose growth doesn't exceed the order 2, minimal type.
Since  $1 \in\mathcal{L-P},$  we have
$$
M_1(z)+M_2(z)=e^{a z^2}(M(z)+\overline{M(\bar{z})}) \in \mathcal{L-P},
$$
where the
growth of $M(z)+\overline{M(\bar{z})}$ doesn't exceed the order 2 and is of  minimal type. Using \eqref{e2}
we conclude that this can be possible only if $a\leqslant 0$.
Thus, if the operator~\eqref{e3} preserve the class $\mathcal{L-P}$, the conditions
$1)$--$3)$ hold for its coefficients $M_1(z)$ and $M_2(z)$.

Conversely, suppose that the conditions
$1)$--$3)$ hold for the operator~\eqref{e3}. Consider any entire function $f\in \mathcal{L-P}$. The condition  $2)$
provides all the assumptions of Theorem \ref{th:mth1} for the function~$\dfrac{M_2}{M_1}$. Therefore, the function
$f(z+i)+\dfrac{M_2(z)}{M_1(z)}f(z-i)$ has only real roots or is identically zero.
The condition~$1)$ implies that the combination $M_1(x)f(x+i)+M_2(x)f(x-i)$ is real for all $x\in \mathbb{R}$.
Now the condition $3)$ gives us the necessary estimate on the growth of
this combination. Thus,
$$
M_1(z)f(z+i)+M_2(z)f(z-i)\in\mathcal{L-P},
$$
as required.
\end{proof}

\setcounter{equation}{0}

\section{Proof of Theorem \ref{th:mth2} }\label{section:proof.thm1.2}

Following notations used in~\cite{lev}, by $\mathcal{HB}$  we denote the class of all entire functions $M(z)$ with no roots in
the closed lower half-plane $\Im z\leqslant 0$ satisfying the condition
$$
\left| \frac{M(z)}{\overline{M(\bar{z})}} \right|< 1,
\quad \Im z>0,
$$
and by $\overline{\mathcal{HB}}$  we denote the class of all entire functions $M(z)$ not vanishing
in the open lower half-plane $\Im z<0$ satisfying
the condition
$$
\left| \frac{M(z)}{\overline{M(\bar{z})}} \right|\leqslant 1, \quad \Im z>0.
$$
Obviously, the equality is possible only if $M(z)=CF(z),$ where $C$ is a non-zero complex constant and $F(z)$ is a real function.

We need the notation: if $f(z)=\sum\limits_{k=0}^{\infty} (a_k+ib_k)z^k,$ where $a_k$ and $b_k$ are real numbers,
then
$$
Rf(z)=\sum\limits_{k=0}^{\infty} a_k z^k\qquad\text{and}\qquad If(z)=\sum\limits_{k=0}^{\infty} b_k z^k.
$$

According to  M.\,Krein's theorem (see~\cite{Akhiezer_Krein} and \cite[p. 411]{lev}), an entire function $\Phi$ belongs to the class
$\mathcal{HB}$ if, and only if, it can be represented in the form
\begin{equation}\label{t5}
\Phi (z)= C z^{n} e^{u(z)+i\beta z}\prod_{k=1}^\infty \left(1-\frac{z}{\alpha_k}\right)
e^{RP_k \left(  \frac{z}{\alpha_k} \right) },
\end{equation}
where $n\in\mathbb{N}\cup\{0\}$, $ \beta \geqslant 0$, $u(z)$ is a real entire function, and $\displaystyle\prod\limits_{k=1}^\infty \left(1-\dfrac{z}{\alpha_k}\right)
e^{P_k \left(  \frac{z}{\alpha_k} \right) }$ is the Hadamard product relating to zeros of the function $\Phi (z)$ (see, e.g.~\cite[Section 2.7]{B},~\cite[Chapter I, Section 10] {lev}, or \cite[Chapter II]{GO}).

\begin{proof}[Proof of Theorem \ref{th:mth2}] The first assertion of Theorem \ref{th:mth2} coincides with the
first one in Theorem \ref{th:mth3}.  So it suffices to prove equivalence of the second assertion of Theorem \ref{th:mth2}
to the second and third assertions of Theorem \ref{th:mth3}.

Obviously, $\dfrac{|M_2(z)|}{|M_1(z)|}=\dfrac{|M_2(z)|}{|\overline{M_2(\bar{z})}|}=1$ if, and only if,
$M_2(z)=C\Phi(z),$ where $C$ is a non-zero complex constant and $\Phi(z)$ is a real entire function.  In this case,
the statement of Theorems \ref{th:mth2} and \ref{th:mth3} are equivalent.

\vspace{1mm}

Now assume that $\left|\dfrac{M_2(z)}{M_1(z)}\right|=\left|\dfrac{M_2(z)}{\overline{M_2(\bar{z})}}\right|< 1$,
$\Im z>0$, that is $M_2(z) \in \overline{\mathcal{HB}}$. All common zeros of the functions $M_2(z) $
and $\overline{M_2(\bar{z})}$ are real. Denote by $(x_k)_{k=1}^{+\infty}$ all real zeros of $M_2(z)$.
Then
\begin{equation}
\label{t6} M_2(z)=\Psi(z)\prod_{k=1}^{\infty} \left(1-\dfrac{z}{x_k}\right)e^{\tfrac{z}{x_k}},
\end{equation}
where $\Psi(z)\in \mathcal{HB}$.  Conversely, if the function $M_2$ is of the form~\eqref {t6}, then
$M_2(z)\in \overline{\mathcal{HB}}$.  By virtue of Theorem \ref{th:mth2} the function $\Psi$  is of the form
$$
\Psi(z)=C z^{n} e^{-az^2+bz}\cdot\prod_{k=1}^{\infty} \left(1-\dfrac{z}{\alpha_k}\right)e^{\tfrac{z}{\alpha_k}},
$$
where $\Im\alpha_k >0$,   $\sum\limits_{k=1}^{\infty}|\alpha_k|^{-2} <\infty$, $ a \geqslant 0$, $C,b\in\mathbb{C}$,  $n\in\mathbb{N}\cup\{0\}$.

Taking into account the Blaschke condition \eqref{f18} we can rewrite this formula in the following way
$$
\Psi(z)=C z^{n} e^{-az^2+\alpha z+iz\left(\beta-\sum\limits_{k=1}^{\infty}\tfrac{\Im\alpha_k}{|\alpha_k|^2}\right)}\cdot\prod_{k=1}^{\infty} \left(1-\frac{z}{\alpha_k}
\right)e^{z \Re\tfrac{1}{\alpha_k}},
$$
where $b=\alpha+i\beta$, see~\eqref{t4}.

By M.\,Krein's theorem mentioned above (see~\eqref{t5}), we have $\Psi(z)\in \mathcal{HB} \Longleftrightarrow$
\begin{equation*}
\Im b -\sum_{k=1}^{\infty} \frac{\mbox{Im}\
\alpha_k}{|\alpha_k|^2} \geqslant 0.
\end{equation*}

If, additionally, $\sum\limits_{k=1}^{\infty}|\alpha_k|^{-1}<\infty$, then
the function $\Psi(z)$ is of genus $0$, so it has the form (see e.g.~\cite[p.???]{lev})
$$
\Psi(z)=C z^{n} e^{-az^2+b z}\prod_{k=1}^{\infty} \left(1-\frac{z}{\alpha_k}
\right),
$$
and from~\eqref{t5} we have
$$
\Im b\geqslant 0,
$$
as required.
\end{proof}

\section{Conclusion}\label{Section:conclusion}

We established necessary and sufficient conditions for the operator~\eqref{e1}
to preserve the Laguerre-P\'olya class $\mathcal{L-P}$ of entire functions. This
can help to define whether a given entire function belongs to $\mathcal{L-P}$
provided we established that this function is a result of the action of
the operator~\eqref{e1} to a function in the class $\mathcal{L-P}$. Proposition~\ref{Proposition1} follows from Theorem~\ref{th:mth2} applied to polynomials.

\section{Acknowledgements}

The second author is Shanghai Oriental Scholar whose research was supported by Russian Science Foundation, grant no. 14--11--00022.

\end{document}